\documentclass{amsart}
\title[Bounds on eigenvalues for Harper's operators]{Useful bounds on the extreme eigenvalues and vectors of matrices for Harper's operators}

\author[Bump]{Daniel Bump}
\address{Department of Mathematics\\ Stanford University\\ 450 Serra Mall, Bldg. 380\\ Stanford,
CA 94305-2125, USA
}
\thanks{The first, second, third, and fifth authors would like to acknowledge partial support from NSF grants DMS 1001079, DMS 08-04324, DMS 1303761, and DMS 1400248 (respectively).  The remaining author would like to acknowledge partial support from ANR grant number ANR-12-BS01-0019.}
\author[Diaconis]{Persi Diaconis}
\address{Department of Mathematics\\ Stanford University\\ 450 Serra Mall, Bldg. 380\\ Stanford,
CA 94305-2125, USA
}
\author[Hicks]{Angela Hicks}
\address{Department of Mathematics\\ Stanford University\\ 450 Serra Mall, Bldg. 380\\ Stanford,
CA 94305-2125, USA
}
\author[Miclo]{Laurent Miclo}
\address{Institut de Math\'{e}matiques de Toulouse\\ Universit\'{e} Paul Sabatier\\ 118 route de Narbonne\\ F-31062 Toulouse {Cedex 9}, France.}
 \author[Widom]{Harold Widom} \address{UC Santa Cruz\\ Department of Mathematics\\ Santa Cruz, CA 95064, USA} 

\newcommand{\nF}{\mathcal{F}_n}
\usepackage{enumerate}
\newtheorem{proposition}{Proposition}
\newtheorem{lemma}{Lemma}
\newtheorem{theorem}{Theorem}
\theoremstyle{definition}
\newtheorem*{remark}{Remark}
\newtheorem*{remarks}{Remarks}
\newtheorem{example}{Example}
\newtheorem{corollary}{Corollary}
\newtheorem*{acknowledgment}{Acknowledgment}
\usepackage{amsmath,amssymb,tikz,accents}
\usetikzlibrary{matrix, arrows,positioning}
\DeclareTextAccent{\myacc}{T1}{4}

\begin{document}

\begin{abstract}
In analyzing a simple random walk on the Heisenberg group we encounter the problem of bounding the extreme eigenvalues of an $n\times n$ matrix of the form $M=C+D$ where $C$ is a circulant and $D$ a diagonal matrix.  The discrete Schr{\"o}dinger operators are an interesting special case.  The Weyl and Horn bounds are not useful here.  This paper develops three different approaches to getting good bounds.  The first uses the geometry of the eigenspaces of $C$ and $D$, applying a discrete version of the uncertainty principle.  The second shows that, in a useful limit, the matrix $M$ tends to the harmonic oscillator on $L^2(\mathbb{R})$ and the known eigenstructure can be transferred back.  The third approach is purely probabilistic, extending $M$ to an absorbing Markov chain and using hitting time arguments to bound the Dirichlet eigenvalues.  The approaches allow generalization to other walks on other groups.
\end{abstract}
\maketitle
\begin{center}\today\end{center}
\section{Introduction}\label{PS1}  Consider the $n\times n$ matrix 
\begin{equation}\label{Peq1.1}M_n=\frac{1}{4}\begin{tikzpicture}[baseline=(current bounding box.center)]
\matrix (m) [matrix of math nodes,nodes in empty cells,right delimiter={)},left delimiter={(} ]{
2  & 1 &  & &&   & 1  \\
 1& & & & &&   \\
  & & & & &&   \\
  & & 1&\phantom{1} &1& &    \\
   & & & & &&   \\
  & & & & &&1  \\
1 & & && &  1&\phantom{1} \\
} ;
\draw[loosely dotted,thick] (m-1-1)-- (m-4-4);
\draw[loosely dotted,thick] (m-2-1)-- (m-4-3);
\draw[loosely dotted,thick] (m-1-2)-- (m-4-5);
\draw[loosely dotted,thick] (m-4-4)-- (m-7-7.center);
\draw[loosely dotted,thick] (m-4-3)-- (m-7-6);
\draw[loosely dotted,thick] (m-4-5)-- (m-6-7);
\node[align=center] (cosgen) at (4.6cm,1cm) [below=3mm]{$2\cos\left(\frac{2\pi j}{n}\right)$,\,\, \tiny{$0\leq j \leq n-1$}};
\path[thick, bend left=45, <-] 
 (m-4-4.center) edge (cosgen.west);
\end{tikzpicture}\end{equation}

As explained in \cite{us} and summarized in Section \ref{PS2}, this matrix arises as the Fourier transform of a simple random walk on the Heisenberg group, as a discrete approximation to Harper's operator in solid state physics and in understanding the Fast Fourier Transform.  Write $M=C+D$ with $C$ a circulant, (having $\frac{1}{4}$ on the diagonals just above and below the main diagonal and in the corners) and $D$ a diagonal matrix (with diagonal entries $\frac{1}{2}\cos\left(\frac{2\pi j}{n}\right)$ for $0\leq j\leq n-1$).  The Weyl bounds \cite{HJ} and Horn's extensions \cite{Bhatia} yield that the largest eigenvalue $\lambda_1(M)\leq \lambda_1(C)+\lambda_1(D)$.  Here $\lambda_1(C)=\lambda_1(D)=\frac{1}{2}$ giving $\lambda_1(M)\leq 1$.  This was not useful in our application; in particular, we need $\lambda_1(M)\leq 1-\frac{\text{const}}{n}$.  This paper presents three different approaches to proving such bounds.
The first approach uses the geometry of the eigenvectors and a discrete version of the Heisenberg uncertainty principle.  It works for general Hermitian circulants:
\begin{theorem}\label{PT1}
 Let $C$ be an $n\times n$ Hermitian circulant with eigenvalues $\lambda_1(C)\geq\dots\geq \lambda_n(C)$.  Let $D$ be an $n\times n $ real diagonal matrix with eigenvalues $\lambda_1(D)\geq \dots\geq \lambda_n(D)$.   If $k,k'$ satisfy $1\leq k,k'\leq n, kk' <n$, then 
 \begin{align*}\lambda_1(C+D)&\leq\lambda_1(C)+\lambda_1(D)\\ &\hspace{.5cm}-\frac{1}{2} \min\{\lambda_1(D)-\lambda_k(D),\lambda_1(C)-\lambda_{k'}(C)\}\left(1-\sqrt{\frac{kk'}{n}}\right)^2.\end{align*}
\end{theorem}
\begin{example}  For the matrix $M_n$ in (\ref{Peq1.1}), the eigenvalues of $C$ and $D$ are real and equal to $\left\{\frac{1}{2} \cos\left(\frac{2\pi j}{n}\right)\right \}_{0\leq j\leq n-1}$.  For simplicity, take $n$ odd.  Then, writing $\lambda_j=\lambda_j(C)=\lambda_j(D)$, $\lambda_1=\frac{1}{2}$, $\lambda_2=\lambda_3=\frac{1}{2} \cos\left(\frac{2\pi}{n}\right)$, and  $\lambda_{2j+1}=\lambda_{2j}=\frac{1}{2}\cos\left(\frac{2\pi j}{n}\right)$ for $1\leq j\leq \frac{n-1}{2}$.
Choose $k=k'=\lfloor c\sqrt{n}\rfloor$ for a fixed $0<c<1$.  Then 
$$\lambda_k=\frac{1}{2}\left(1-\frac{1}{2}\left(\frac{\pi c}{n^{1/2}}\right)^2+O\left(\frac{1}{n^{3/2}}\right)\right), \,\,\, \left(1-\sqrt{\frac{k k'}{n}}\right)^2\geq(1-c)^2$$ and the bound in Theorem \ref{PT1} becomes $$
\lambda_1(M_n)\leq 1-\frac{\pi^2}{8}\frac{c^2(1-c)^2}{n}+O\left(\frac{1}{n^{3/2}}\right).$$
The choice $c=\frac{1}{2}$ gives the best result.  Very sharp inequalities for the largest and smallest eigenvalues of $M_n$ follow from \cite{BZ}.  They get better constants than we have in this example.  Their techniques make sustained careful use of the exact form of the matrix entries while the techniques in Theorem \ref{PT1} work for general circulants.
\end{example}
The second approach passes to the large $n$ limit, showing that the largest eigenvalues of $M_n$ from (\ref{Peq1.1}) tend to suitably scaled eigenvalues of the harmonic oscillator $L=-\frac{1}{4}\frac{d^2}{dx^2}+\pi^2x^2.$
\begin{theorem}\label{PT2} For a fixed $k\geq 1$, the $k$th largest eigenvalue of $M_n$ equals $$1-\frac{\mu_k}{n}+o\left(\frac{1}{n}\right)$$
with $\mu_k=\frac{(2k-1)\pi}{2}$, the $k$th smallest eigenvalue of $L$. \end{theorem} 
Theorem \ref{PT2} gets higher eigenvalues with sharp constants for a restricted family of matrices.  The argument also gives a useful approximation to the $k$th eigenvector.  Similar results (with very different proofs) are in \cite{Strohmer}.

  There are many techniques available for bounding the eigenvalues of stochastic matrices (\cite{LS}, \cite{DS}, and \cite{DSC}).  We initially thought that some of these would adapt to $M_n$.  However, $M_n$ is far from stochastic: the row sums of $M_n$ are not constant and the entries are sometimes negative.  Our third approach is to let  $M_n'=\frac{1}{3}I+\frac{2}{3}M_n$.  This is substochastic (having non-negative entries and row sums at most 1).  If   $a_i=1-\sum_j {M_n'}(i,j)$, consider the   $(n+1)\times(n+1)$ stochastic matrix:

\begin{equation}\label{Peq2} M_n''=\begin{tikzpicture}[baseline=(current bounding box.center)]
\matrix (m) [matrix of math nodes,nodes in empty cells,right delimiter={)},left delimiter={(} ]{
1 &0&0&\dots&0\\
 a_1&&&&\\
  a_2&&&&\\
  \vdots &&&&\\
   a_n&&&&.\\} ;
\node[align=center] (cosgen) at (0.3 cm,.6cm) [below=3mm]{\huge $M_n^\prime$};
\end{tikzpicture}\end{equation}

This has the interpretation of an absorbing Markov chain (0 is the absorbing state) and the Dirichlet eigenvalues of $M_n''$ (namely those whose eigenvalues vanish at 0) are the eigenvalues of $M_n'$.  In \cite{us} path and other geometric techniques are used to bound these Dirichlet eigenvalues.  This results in bounds of the form $1-\frac{\text{const.}}{n^{4/3}}$ for $\lambda_1(M_n)$.  While sufficient for the application, it is natural to want an improvement that gets the right order.  Our third approach introduces a purely probabilistic technique which works to give bounds of the right order for a variety of similar matrices.  
\begin{theorem} \label{PT3}  There is a $c>0$ such that, for all $n\geq1$ and $M_n$ defined at (\ref{Peq1.1}), the largest eigenvalue satisfies  $\lambda_1(M_n)\leq 1-\frac{c}{n}$.
\end{theorem}
Section \ref{PS2} gives background and motivation.  Theorems \ref{PT1}, \ref{PT2}, and \ref{PT3} are proved in Sections \ref{PS3}, \ref{PS4}, and \ref{PS5}.  Section \ref{PS6} treats a simple random walk on the affine group mod $p$.  It uses the analytic bounds to show that order $p^2$ steps are necessary and sufficient for convergence.  It may be consulted now for further motivation.  The final section gives the limiting distribution of the bulk of the spectrum of $M_n(a)$ using the Kac-Murdock-Szeg{\"o} theorem.
\section{Background}\label{PS2}
Our work in this area starts with the finite Heisenberg group:
$$H_1(n)=\left\{\begin{pmatrix}1&x&z\\0&1&y\\0&0&1\end{pmatrix}:\hspace{1cm} x,y,z\in \mathbb{Z}/n\mathbb{Z}\right\}.$$

Write such a matrix as $(x,y,z)$, so $$(x,y,z)(x',y',z')=(x+x',y+y',z+z'+xy').$$ Let \begin{align}S=\{(1,0,0),(-1,0,0),(0,1,0),(0,-1,0)\}\text{ and}\end{align}\label{P1.1}
\begin{align}\label{P1.2}Q(g)=\begin{cases}\frac{1}{4}& g\in S\\0&\text{ otherwise}\end{cases}.
\end{align}
Thus $S$ is a minimal symmetric generating set for $H_1(n)$ and $Q$ is the probability measure associated with `pick an element in $S$ at random and multiply.'  Repeated steps of this walk correspond to convolution. For $(x,y,z)\in H_1(n)$, $$Q^{*^k}(x,y,z)=\sum_{(x',y',z')\in H_1(n)} Q(x',y',z') Q^{*^{k-1}}((x,y,z)(x',y',z')^{-1}).$$  When $k$ is large, $Q^{*^k}$ converges to the uniform distribution $U(x,y,z)=\frac{1}{n^3}$  The rate of convergence of $Q^{*^k}$ to $U$ can be measured by the chi-squared distance:

\begin{align}\label{Peq2.1}  
\sum_{(x,y,z)\in H} |Q^{*^k}(x,y,z)-U(x,y,z)|^2/(U(x,y,z))=\sum_{\substack{\rho\in \hat{H}_1\\ \rho\neq 1}}d_\rho\|\hat{Q}(\rho)^k\|^2,
\end{align}
On the right, the sum is over nontrivial irreducible representations of $H_1(n)$ with $\rho$ of dimension $d_\rho$ and $\hat{Q}(\rho)^k=\sum_{(x,y,z)}Q^{*^k}(x,y,z)\rho(x,y,z)$.  For background on the Fourier analysis approach to bounded convergence see \cite{Diaconis}, Chapter 3.

For simplicity, (see \cite{us} for the general case) take $n=p$ a prime.  Then $H_1(p)$ has $p^2$ 1-dimensional representations $\rho_{a,b}(x,y,z)=e^{\frac{2\pi i}{p}(ax+by)}$ for $a,b$ in $\mathbb{Z}_p$.  It has $p-1$ $p$-dimensional representations.  These act on $V=\{f:\mathbb{Z}_p\rightarrow \mathbb{C}\}$ via
$$\rho_a(x,y,z)f(w)=e^{\frac{2\pi i a}{p}(yw+z)}f(x+w), \,\, 0\leq a\leq p-1.$$  

The Fourier transform of $Q$ at $\rho_a$ is the matrix $M_n(a)$ as in (\ref{Peq1.1}) with $\cos\left(\frac{2\pi j}{p}\right)$ replaced by $\cos\left(\frac{2\pi a j}{p}\right)$ for $0\leq j\leq p-1$.

The chi-squared norm in (\ref{Peq2.1}) is the sum of the ($2k$)th power of the eigenvalues so proceeding needs bounds on these.  The details are carried out in \cite{us}.  The main results show that $k$ of order $n^2$ steps are necessary and sufficient for convergence.  That paper also summarizes other appearances of the matrices $M_n(a)$.  They occur in discrete approximations of the `almost Mathieu' operator in solid state physics.  In particular, see \cite{WPR}, \cite{BZ}, and \cite{BVZ}.  If $\mathcal{F}_n$ is the discrete Fourier transform matrix $\left((\mathcal{F}_n)_{jk}=\frac{1}{\sqrt{n}}e^{\frac{2\pi ijk}{n}}\right)$; it is easy to see that $\mathcal{F}_nM_n(1)=M_n(1)\mathcal{F}_n$.  Diagonalizing $\mathcal{F}_n$ has engineering applications and having a `nice' commuting matrix should help.  For this reason, there is engineering interest in the eigenvalues and vectors of $M_n(1)$. See \cite{DStei} and \cite{Mehta}.

\section{Proof of Theorem \ref{PT1}}\label{PS3}
Throughout this section $C$ is an $n\times n$ Hermitian circulant with eigenvalues $\lambda_1(C)\geq\lambda_2(C)\geq\dots\geq \lambda_n(C)$ and $D$ is a real diagonal matrix with eigenvalues $\lambda_1(D)\geq \lambda_2(D)\geq\dots\geq \lambda_n(D)$.  Let $x$ be an eigenvector of $C+D$ corresponding to $\lambda_1(C+D)$.  Recall that  $\left((\mathcal{F}_n)_{jk}=\frac{1}{\sqrt{n}}e^{\frac{2\pi ijk}{n}}\right)$ for $j,k\in \mathbb{Z}/n\mathbb{Z}$.  This has rows or columns which simultaneously diagonalize all circulants. Write $\hat{x}=\mathcal{F}_nx$ and $x^h$ for the conjugate transpose.  We use $\|x\|^2=x^hx$.

Our aim is to prove that for $kk'<n$,
\begin{align}\label{Peq3.1} 
\lambda_1(C+D)&\leq\lambda_1(C)+\lambda_1(D)\\ &\hspace{.5cm}-\frac{1}{2} \min\{\lambda_1(D)-\lambda_{k+1}(D),\lambda_1(C)-\lambda_{k'+1}(C)\}\left(1-\sqrt{\frac{kk'}{n}}\right)^2.\notag
\end{align}
The first step is to write $x^hCx$ in terms of a Fourier transform pair $\hat{x}=\mathcal{F}_nx$.  A subtle point is that although $\mathcal{F}_n$ diagonalizes $C$, the resulting diagonal matrix does not necessarily have entries in decreasing order, necessitating a permutation indexing in the following lemma. 

\begin{lemma} Define a permutation $\sigma$ such that $$\frac{e^{\frac{2\pi i}{n}\sigma_j^{-1}b}}{\sqrt{n}},\,\, 0\leq b\leq n-1$$ is the eigenvector corresponding to $\lambda_j(C)$. Then  \begin{align} x^hCx=\hat{x}^hD'\hat{x}\end{align} with $D'=\text{diag}(\lambda_{\sigma_1}(C),\dots,\lambda_{\sigma_n}(C))$.
\end{lemma}
\begin{proof}
Since $C$ is diagonalized by $\mathcal{F}_n^h$, $\mathcal{F}_nC\mathcal{F}_n^h=D'$.
\noindent Thus $$x^hCx= x^h\nF^h\nF C\nF^h\nF x=\hat{x}^h D'\hat{x}.$$
\end{proof} 
A key tool is the Donoho-Stark \cite{DonSta} version of the Heisenberg Uncertainty Principle.  For this, call a vector $y$ `$\epsilon$-concentrated on a set $S\subset[n]$' if $|x_i|<\epsilon$ for $i\notin S$.  
\begin{theorem}[Donoho-Stark]  Let $y$, $\hat{y}$ be a unit norm Fourier Transform pair with $y$ $\epsilon_S$-concentrated on $S$ and $\hat{y}$ $\epsilon_T$-concentrated on $T$.  Then \begin{equation}\label{Peq2.2}|S||T|\geq n (1-(\epsilon_S+\epsilon_T))^2.\end{equation}\end{theorem}  
Let $(y)_S$ be the projection onto the subspace vanishing off $S$:
$$((y)_S)_i=\begin{cases}y_i &i\in S\\0 &\text{otherwise}\end{cases}.$$
A simple consequence of the bound (\ref{Peq2.2}) is \begin{corollary}\label{AngCor} If $kk'<n$, $z,\hat{z}$ a unit norm Fourier transform pair and $S$ and $T$ are sets of size $k$, $k'$, then $$\|(z)_{S^c}\|^2+\|(\hat{z})_{T^c}\|^2\geq\frac{1}{2}\left(1-\sqrt{\frac{kk'}{n}}\right)^2.$$
\end{corollary}
\begin{proof}  Let $\epsilon_S=\|(z)_{S^c}\|$ and $\epsilon_T=\|(\hat{z})_{T^c}\|$.  Then $\|z-(z)_S\|=\|(z)_{S^c}\|$ and  $\|\hat{z}-(\hat{z})_T\|=\|(\hat{z})_{T^c}\|.$  Thus $z$ is $\epsilon_S$ concentrated on $S$ and $\hat{z}$ is $\epsilon_T$ concentrated on $T$.  Thus (\ref{Peq2.2}) gives 
$$\frac{|S||T|}{n}=\frac{kk'}{n}\geq(1-(\epsilon_S+\epsilon_t))^2\text{ or }(\epsilon_S+\epsilon_T)\geq 1-\sqrt{\frac{kk'}{n}},$$ so if $kk'\leq n$
  $$\|(z)_{S^c}\|^2+\|(\hat{z})_{T^c}\|^2= (\epsilon_S^2+\epsilon_T^2)\geq\frac{1}{2}(\epsilon_s+\epsilon_T)^2\geq\frac{1}{2}\left( 1-\sqrt{\frac{kk'}{n}}\right)^2$$
 
\end{proof}
\begin{proof}[Proof of Theorem \ref{PT1}] With notation as above, 
$$\lambda_1(C+D)=x^h(C+D)x=x^hCx+x^hDx=\hat{x}^hD'\hat{x}+x^hDx=:*.$$
Let $\overline{D}=D-\lambda_1(D)I$ and $\overline{D'}=D'-\lambda_1(C)I$.
Then \begin{align*}*&=\hat{x}^h\lambda_1(C)I\hat{x}+\hat{x}^h\overline{D'}\hat{x}+x^h\lambda_1(D)Ix+x^h\overline{D}x
\\&=\lambda_1(C)+\hat{x}^h\overline{D'}\hat{x}+\lambda_1(D)+x^h\overline{D}x\end{align*}

Now $\overline{D}$ and $\overline{D'}$ have non-positive eigenvalues so our improvement over the Weyl bounds will follow by showing that $x$ or $\hat x$ have support on suitably negative entries of $\overline{D}'$ or $\overline{D}$.

Let $S$ and $T$ correspond to the largest $k$, $k'$ entries of $\overline{D}$, $\overline{D}'$, respectively. Then $x=(x)_S+(x)_{S^c}$, $\hat{x}=(\hat{x})_T+(\hat{x})_{T^c}$.  Each of those decomposition is into orthogonal pieces.  Multiplying any of the four pieces by an arbitrary diagonal matrix preserves this orthogonality. Thus
$$*=\lambda_1(C)+\lambda_1(D)+(\hat{x})_T^h\overline{D'}(\hat{x})_{T}+(\hat{x})_{T^c}^h\overline{D'}(\hat{x})_{T^c}+(x)_S^h\overline{D}(x)_S+(x)_{S^c}^h\overline{D}(x)_{S^c}.$$
For the last four terms on the right, terms 1 and 3 are bounded above by zero and 2 and 4 contribute with the following bounds:
\begin{align*}*&\leq \lambda_1(C)+\lambda_1(D)+(\lambda_{k+1}(D)-\lambda_1(D))\|(x)_{S^c}\|^2+(\lambda_{k'+1}(C)-\lambda_1(C)\|(\hat{x})_{T^c}\|^2
\\&\leq \lambda_1(C)+\lambda_1(D)\\&\hspace{1.3cm}+\min\{(\lambda_{k+1}(D)-\lambda_1(D)),(\lambda_{k'+1}(C)-\lambda_1(C)\}\left(\frac{1}{2}\left( 1-\sqrt{\frac{kk'}{n}}\right)^2\right)
\end{align*}
where the last line follows from the corollary.\end{proof}
\begin{remarks}
\begin{enumerate}
\item These arguments work to give the smallest eigenvalue as well, so in fact we also have for $ll'<n$:
\begin{align}
\lambda_n(C+D)&\geq\lambda_n(C)+\lambda_n(D)\\ &\hspace{.5cm}+\frac{1}{2} \min\{\lambda_{n-l}(D)-\lambda_{n}(D),\lambda_{n-l'}(C)-\lambda_{n}(C)\}\left(1-\sqrt{\frac{ll'}{n}}\right)^2.\notag
\end{align}
\item Our thanks to a thoughtful anonymous reviewer, who pointed out that Corollary \ref{AngCor} can be improved using Cauchy- Schwartz to show that for $0<a,b\leq 1$,
$$\left( 1-\sqrt{\frac{kk'}{n}}\right)^2\leq (\epsilon_s+\epsilon_T)^2\leq (a\epsilon_S^2+b\epsilon_T^2)\left(\frac{1}{a}+\frac{1}{b}\right).$$
Setting $a=\lambda_{k+1}(D)-\lambda_1(D)$ and $b=\lambda_{k+1}(D)-\lambda_1(D)$, one can improve the previous theorem:
\begin{equation*}\begin{split}
\lambda_1(C+D) & \le \lambda_1(C)+\lambda_1(D)\\
& -\frac{(\lambda_1(D)-\lambda_{k+1}(D))(\lambda_1(C)-\lambda_{k^{\prime}+1}(C))}{(\lambda_1(D)-\lambda_{k+1}(D)) + (\lambda_1(C)-\lambda_{k^{\prime}+1}(C))}
\left(1-\sqrt{\frac{kk^{\prime}}{n}}\right)^2.
\end{split}\end{equation*}
In our case, the result is the same, since the eigenvalues of $C$ and $D$ are identical.
\item Donoho and Stark \cite{DonSta} give many variations on their uncertainty principle suitable for other transforms.  The techniques above should generalize, at least to the $G$-circulants of \cite{DPatt}. 
\item There should be similar theorems with $C$ and $D$ replaced by general Hermitian matrices and perhaps extensions to higher Weyl and Horn inequalities (see \cite{Bhatia} and \cite{Fulton}).
\item Further applications/examples are in Section \ref{PS6}.
\end{enumerate}
\end{remarks}
\def\ov{\over} \def\be{\begin{equation}} 
\def\ee{\end{equation}} \def\R{\mathbb R} \def\Z{\mathbb Z}
\def\ph{\phi} \def\inv{^{-1}} \def\d{\delta}
\def\bc{\begin{center}} \def\ec{\end{center}} \def\iy{\infty}
\newcommand{\ch}{\raisebox{.2ex}{$\chi$}} \def\Mnt{\tilde{M_n}} \def\l{\ell}
\def\sqn{\sqrt n} \def\sqni{{1\ov\sqn}} \def\la{\lambda} \def\iy{\infty}
\def\({\left(} \def\){\right)} \def\D{\Delta} \def\wh{\widehat} 
\def\ep{\varepsilon} \def\m{\mu} \def\sp{\vspace{1ex}} \def\noi{\noindent}
\def\C{\overline{C}} \def\a{\alpha} \def\ov{\over}

\section{The Harmonic oscillator as a limit.} \label{PS4}

We prove Theorem 2, that for $k\ge1$ the $k$th largest eigenvalue of $M_n$ is equal to $1-\mu_k/n+o(1/n)$ and the $k$th smallest eigenvalue of $M_n$ is equal to $-1+\mu_k/n+o(1/n)$, where $\mu_k$ is the $k$th smallest eigenvalue of 
\[L=-{1\ov4}\,{d^2\ov dx^2}+\pi^2\,x^2\]
on $(-\iy,\iy)$.  By a classical computation (see \cite{Griffiths}), $\mu_k=(2k-1)/(2\pi).$

The $n\times n$ matrix $M_n$ has $j,k$-entry
\[{1\ov4}\,[\d(j-k-1)+\d(j-k+1)]+{1\ov2}\cos(2\pi k/n)\,\d(j-k),\]
where $j,k\in\Z_n=\Z/n\Z$. 

We define
\[\Mnt=n\,(I-M_n).\]
This has $j,k$ entry $m_1(j,k)+m_2(j,k)$, where
\be m_1(j,k)={n\ov2}\,\(\d(j-k)-{1\ov2}\,[\d(j-k-1)+\d(j-k+1)]\)\label{m1}\ee
\be m_2(j,k)={n\ov2}\,(1-\cos(2\pi k/n))\,\d(j-k).\label{m2}\ee

We will show first that if $\m$ is any limit of eigenvalues of $\Mnt$ then $\m$ is an eigenvalue of $L$; and, second, that any eigenvalue $\m$ of $L$ has a neighborhood that contains exactly one eigenvalue, counting multiplicity, of $\Mnt$ for $n$ sufficiently large. These imply the stated result.  

These will be accomplished as follows. Give each point of $\Z_n$ measure $1/\sqn$, so the total measure equals $
\sqn$. We then define an isometry $T$ from $L^2(\Z_n)$ to $L^2(-\sqn/2,\sqn/2)$ (thought of as a subspace of $L^2(\R)$ with Lebesgue measure) for which the following hold:

\begin{proposition}\label{HP1} Suppose $\{u_n\}$ is a sequence of functions of norm one in $L^2(\Z_n)$ such that the sequence $\{(\Mnt\,u_n,u_n)\}$ of inner products is bounded. Then $\{Tu_n\}$ has a strongly (i.e., in norm) convergent subsequence.
\end{proposition}
\begin{proposition}\label{HP2} If $\ph$ is a Schwartz function on $\R$ then $T\Mnt T^*\ph\to L\ph$ strongly.\footnote{The operator $T^*$ acts on $L^2(-\sqn/2,\sqn/2)$, so $\ph$ is first to be restricted to $(-\sqn/2,\sqn/2)$.}
\end{proposition}
These will easily give the desired results. (See Propositions \ref{HP3} and \ref{HP4} near the end.)  The final section \ref{HS2} treats the smallest eigenvalues.
\vspace{.1ex}
\subsection{Proofs for the largest eigenvalues}\label{PS4.1}
We use two transforms (with, confusingly, the same notation). First, for $\ph$ in $L^2(-\sqn/2,\sqn/2)$ we define
\[\wh\ph(\l)=\int_{-\sqn/2}^{\sqn/2}e^{-2\pi i \l x/\sqn}\,\ph(x)\,dx,\ \ \ \  (\l\in\Z),\]
and we have by Parseval (after making the substitution $x\mapsto x\sqn$ in the integral) 
\be\|\wh\ph\|=n^{1/4}\,\|\ph\|.\label{norm1}\ee
Here $\|\wh\ph\|^2=\sum_{\l\in\Z} |\hat{\ph}(\l)|^2$.
\sp 

For $u\in L^2(\Z_n)$ we have its finite Fourier transform
\[\wh u(\l)=\sum_k e^{-2\pi i\l k/n}\,u(k),\ \ \ \ (\l\in\Z_n),\]
and we compute below that
\be\|\wh u\|=n^{3/4}\,\|u\|.\label{norm2}\ee
Here 
\[\|\wh u\|^2=\sum |\wh u(\l)|^2,\]
the sum over any integer interval of length $n$. To show (\ref{norm2}), we have
\[|\wh u(\l)|^2=\sum_{j,k}e^{-2\pi i \l (j-k)/n}\,u(j)\,\overline{u(k)}.\]
Since $e^{-2\pi i (j-k)/n}$ is an $n$th root of unity, equal to 1 only when $j=k$ in $\Z_n$, we get
\[\|\wh u\|^2=\sum_\l |\wh u(\l)|^2=n\,\sum_{k}|u(k)|^2=n^{3/2}\,\|u\|^2.\]
\sp

Now we define the operator $T$. Let $J$ be an interval of integers of length $n$ (which later will be specified further) and set
\[D_n(x)=\sum_{\l\in J}e^{2\pi i\l x}.\]
Then $T$ is defined by
\[(Tu)(x)={1\ov n}\sum_k D_n\({x\ov\sqn}-{k\ov n}\)\,u(k).\]
Thus $T$ has kernel 
\[T(x,k)={1\ov n}\,D_n\({x\ov\sqn}-{k\ov n}\).\]
By the definition of the inner product on $L^2(\Z_n)$ we find that $$T^*:L^2(-\sqn/2,\sqn/2)\to L^2(\Z_n)$$ has kernel
\[T^*(k,x)={1\ov\sqn}\,D_n\({k\ov n}-{x\ov\sqn}\).\]

In terms of the transforms we have the following:

\begin{lemma}\label{HL1}
\noi (a) For $u\in L^2(\Z_n)$,
\[\wh{Tu}(\l)=\left\{\begin{array}{ll}\sqni\,\wh u(\l)&{\rm if}\ \l\in J,\\&\\0&{\rm if}\ \l\not\in J.\end{array}\right.\]

\noi(b) For $\ph\in L^2(-\sqn/2,\sqn/2)$,
\[\wh{T^*\ph}(\l)=\sqn\;\wh\ph(\l)\ \ {\rm when}\ \l\in J.\]
\end{lemma}
\begin{proof}  For (a), we have
\begin{align*}\wh{Tu}(\l)&={1\ov n}\,\int \sum_{k,\,\l'\in J}e^{-2\pi i \l x/\sqn}\,
e^{2\pi i \l'(x/\sqn-k/n)}\,u(k)\,dx\\&={1\ov n}\,\int \sum_{\l'\in J}
e^{2\pi i (\l'-\l)x/\sqn}\,\wh u(\l')\,dx.\end{align*}
The result follows.\sp

For (b), we have when $\l\in J$,
\begin{align*}\wh{T^*\ph}(\l)&=\sqni\,\sum_{k,\,\l'\in J}\int e^{-2\pi i\l k/n}\,
e^{2\pi i \l'(k/n-x/\sqn)}\,\ph(x)\,dx\\&=\sqn \int e^{-2\pi i \l x/\sqn}\,\ph(x)\,dx=\sqn\;\wh\ph(\l).\end{align*}
\end{proof}

We show two things about $T$. For the second we shall assume now and hereafter that the end-points of $J$ are $\pm n/2+O(1)$, although this is a lot stronger than necessary.

\begin{lemma}\label{HL2} (a) $T^*\,T=I$. (b) $T\,T^*\to I$ strongly as $n\to\iy$.
\end{lemma}
\begin{proof} By Lemma \ref{HL1}b,
\[\wh{T^*Tu}(\l)=\sqn\;\wh{Tu}(\l'),\]
where $\l'\in J$ and $\l'-\l\in n\mathbb{Z}$. By Lemma \ref{HL1}a this in turn equals $\wh u(\l'),$ which equals $\wh u(\l)$ since $\wh u$ is $n$-periodic. This gives (a).

For (b) observe that $T\,T^*$ is self-adjoint. Since $(T\,T^*)^2=T\,T^*\,T\,T^*=T\,T^*$, it is a (nonzero) projection and so has norm one. Therefore it suffices to show that if $\ph$ is a Schwartz function then $T\,T^*\ph\to\ph$. We have from Lemma \ref{HL1}a that
\[\wh{T\,T^*\ph}(\l)=\sqni\,\wh{T^*\ph}(\l)\]
if $\l\in J$ and equals zero otherwise. If $\l\in J$ then by Lemma \ref{HL1}b it equals $\wh\ph(\l)$. It follows that
\[\|\wh{T\,T^*\ph}-\wh\ph\|^2=\sum_{\l\not\in J}|\wh\ph(\l)|^2.\]
Integrating by parts shows that
\[\wh\ph(\l)=O(\sqn/\l),\]
and so, by our assumption on $J$, the sum on the right side is $O(1)$. Then by (\ref{norm1}) we get
\[\|T\,T^*\ph-\ph\|=O(n^{-1/4}).\]
\end{proof}
Now the work begins. First, an identity. We introduce the notations
\[\C(\xi)=1-\cos(2\pi\xi),\ \ \ S(\xi)=\sin(\pi\xi),\]
and observe that $\C(\xi)=2\,S(\xi)^2$.
\sp

\begin{lemma}\label{HL3} For $u\in L^2(\Z_n)$,
\[(\Mnt\,u,\,u)=n^{-1/2}\,\|S(\l/n)\,\wh u(\l)\|^2+n\,\|S(k/n)\,u(k)\|^2.
\]
\end{lemma}
\noindent{\it Note.}\ Here and below we display ``$k$'' as the variable in the ambient space $\Z_n$ and ``$\l$'' as the variable in the space $\Z_n$ of the Fourier transform. We abuse notation and, for example, the ``$u(k)$'' above denotes the function $k\to u(k)$.
\begin{proof} We consider first the contribution of (\ref{m1}) to the inner product. If we define the operators $A_\pm$ by
\[(A_\pm u)(k)=u(k\pm 1),\]
we see that the contribution to the inner product is
\[{n\ov2}\,(u-[A_+u+A_-u]/2,u).\]
Now
\[\wh{A_\pm u}(\l)=e^{\pm 2\pi i\l/n}\,\wh u(\l),\]
so if we use (\ref{norm2}) we see that the above is equal to
\[{1\ov2}n^{-1/2}\,(\C(\l/n)\,\wh u,\wh u(\l))=n^{-1/2}\,\|S(\l/n)\,\wh u\|^2.\]

To complete the proof of the lemma we note that the contribution to the inner product of (\ref{m2}) is clearly
\[{n\ov2}\,(\C(k/n)\,u(k),u(k))=n\,\|S(k/n)\,u(k)\|^2.\]
\end{proof}
\begin{lemma}\label{HL4} Suppose $u_n$ satisfy $(\Mnt u_n,u_n)=O(1)$. Then 
(a) $\|x\,Tu_n(x)\|=O(1)$, and (b) $\|(Tu_n)'\|=O(1).$\end{lemma}
\begin{proof}[Proof of (a)] We have 
\[\wh u_n(\l)-\wh u_n(\l+1)=\sum_k e^{-2\pi \l k/n}(1-e^{-2\pi i k/n})\,u_n(k),\]
the finite Fourier transform of $(1-e^{-2\pi i k/n})\,u_n(k)$.
We have,
\[|(1-e^{-2\pi i k/n})\,u_n(k)|=2|S(k/n)\,u_n(k)|.\]
Therefore from (\ref{norm2}) and 
\[\|S(k/n)\,u_n(k)\|=O(n^{-1/2}),\]
which follows from Lemma \ref{HL3}, we get
\[\|\wh u_n(\l)-\wh u_n(\l+1)\|=O(n^{1/4}).\]
It follows from Lemma \ref{HL1}a that
$\wh{Tu_n}(\l)=\wh{Tu_n}(\l+1)=0$
if both $\l,\l+1\not\in J$ and
\be\sum_{\l,\,\l+1\in J}|\wh{Tu_n}(\l)-\wh{Tu_n}(\l+1)|^2=O(n^{-1/2}).\label{Tuest}\ee
If $\l\in J$ but $\l+1\not\in J$ then $\wh{Tu_n}(\l+1)=0$ and $\l$ is the right end-point of $J$ and therefore $n/2+O(1)$. From
\be\|S(\l/n)\,\wh u_n(\l)\|=O(n^{1/4}),\label{Sln}\ee
which also follows from Lemma \ref{HL3}, and that $|S(\l/n)|$ is bounded below for $\l=n/2+~O(1)$, we have in particular that $\wh u_n(\l)=O(n^{1/4})$, Therefore 
$\wh{Tu_n}(\l)=O(n^{-1/4})$.
 
So the bound in (\ref{Tuest}) holds when the sum is taken over all $\l\in\Z$. Since 
\[\wh{Tu_n}(\l)-\wh{Tu_n}(\l+1)=\int e^{-2\pi i x\l/\sqn}(1-e^{-2\pi i x/\sqn})\,Tu_n(x)\,dx,\]
it follows from (\ref{norm1}) that
\[\|S(x/\sqn)\,Tu_n(x)\|=O(n^{-1/2}).\]
In the interval of integration $|x|<\sqn/2$, so $|S(x/\sqn)|$ is bounded below by a constant times $|x|/\sqn$. This gives (a).
\end{proof}

\begin{proof}[Proof of (b)] We have 
\[(Tu_n)'(x)={2\pi i\ov n^{3/2}}\sum_{k,\,\l\in J}\l\,e^{2\pi i\l(x/\sqn-k/n)}\,u_n(k)={2\pi i\ov n^{3/2}}\sum_{\l\in J}\l\,e^{2\pi i\l x/\sqn}\,\wh{u_n}(\l).\]
Thus $\wh{Tu_n'}(\l)=2\pi i\l\,\wh u_n(\l)/n$ for $\l\in J$, and it follows from (\ref{norm1}) that
\[\|(Tu_n)'\|^2={4\pi^2\ov n^{5/2}}\sum_{\l\in J}\l^2\,|\wh{u_n}(\l)|^2.\]
Now $|S(\l/n)|$ is bounded below by a constant times $|\l/n|$ for $\l\in J$, so (\ref{Sln}) implies that
\[\sum_{\l\in J}\l^2\,|\wh {u_n}(\l)|^2=O(n^{5/2}),\]
which gives the result.
\end{proof}

\begin{proof}[Proof of Proposition \ref{HP1}] Since $T$ is an isometry each $\|Tu_n\|=1$, and by passing to a subsequence we may assume $\{Tu_n\}$ converges weakly to some $f\in L^2(\R)$. We use the fact that strong convergence will follow if we can show that $\|f\|\ge1$. 
(In general, if $\|f_n\|=1$ and $f_n\to f$ weakly, then $\|f\|\ge1$ implies that $f_n\to f$ strongly. Here is the argument.  We have that  \[\|f_n-f\|^2=\|f_n\|^2+\|f\|^2-2\,{\rm Re}\,(f_n,f).\]
By weak convergence, $(f_n,f)\to \|f\|^2$. Therefore
\[\|f_n-f\|^2\to 1-\|f\|^2\le 0,\]so $\|f_n-f\|\to 0.$)

The hypothesis of Lemma \ref{HL4} is satisfied. It follows from Lemma \ref{HL4}a that for each $\ep>0$ there is a  bounded interval $A$ such that $$\|(1-\ch_A)\,Tu_n\|\le \ep$$ for all $n$. So $\|\ch_A\,Tu_n\|\ge1-\ep.$ It follows from Lemma \ref{HL4}b that $\{Tu_n\}$ is equicontinuous, and this combined with $\|\ch_A\,Tu_n\|\le 1$ shows that a subsequence of $\{Tu_n\}$ converges uniformly on $A$ (to $\ch_A\,f$), and so 
$\|\ch_A\,f\|\ge 1-\ep$. Thus, $\|f\|\ge 1$.
\end{proof}

\begin{proof}[Proof of Proposition \ref{HP2}] Consider first the operator corresponding to $m_1$ in (\ref{m1}). We call it $n\D_k^2$. (The subscript indicates that it acts on functions of $k$.) We show first that 
\[n\,T\,\D_k^2\, T^*\ph\to-\ph''/4\]
in $L^2(\R)$. We have
\begin{align*}
n\,(T\, \D_k^2\, T^*\ph)(x)&\\
&\hspace{-1cm}={1\ov 2n^{1/2}}\sum_{\l,\l',k}
e^{2\pi i\,\l((x/\sqn-k/n)}\int_{-\sqn/2}^{\sqn/2} \D_k^2\,
e^{2\pi i\,\l'((k/n-y/\sqn)}\,\ph(y)\,dy\end{align*}
The exponent in the integral is a function of $k/\sqn-y$. So taking the second difference $\D_k^2$ in $k$ is the same as taking the second difference $\D_y^2$ in $y$ as long as the differences in the $y$-variable are $1/\sqn$. With this understanding, the above equals
\[{1\ov 2n^{1/2}}\sum_{\l,\l',k}
e^{2\pi i\,\l((x/\sqn-k/n)}\int_{-\sqn/2}^{\sqn/2} \D_y^2\,
e^{2\pi i\,\l'((k/n-y/\sqn)}\,\ph(y)\,dy.\] 
By changing variables in two of the three summands from $\D_y^2$ we can put the $\D_y^2$ in front of the $\ph(y)$. There is an error because of the little change of integration domains but (for $\ph$ a Schwartz function) this is a rapidly decreasing function of $n$, and so can be ignored. After this what we get is $nT\,T^*\,\D^2\,\ph$. Taylor's theorem gives
\[n(\D_x^2\,\ph)(x)=-\ph''(x)/4+O\(n^{-1/2}\,\max_{|y-x|<1/\sqn}|\ph'''(y)|\),\]
from which it follows that $n\,(\D_x^2\,\ph)(x)\to -\ph''(x)/4$ strongly. Since $T\,T^*\to I$ strongly we deduce that $n\,T\, \D_k^2\, T^*\ph\to-\ph''/2$ strongly.

Lastly, consider the operator corresponding to (\ref{m2}), which is multiplication by $n\,\C(k/n)/2=n(1-\cos(2\pi k/n))/2$. For convenience we call this operator $\C_n/2$. 

By Lemma \ref{HL1}b we know that $\wh{T^*\ph}(\l)=\sqn\,\wh\ph(\l)$ when  $\l\in J$.
In general,
\[\wh{\C_n\,u}(\l)=n\,\sum_k (1-\cos(2\pi k/n))\,e^{-2\pi ik\l/n}\,u(k)=
n\,(\D^2\,\wh u)(\l),\]
where here
\[(\D^2\,\wh u)(\l)=\wh u(\l)-[\wh u(\l-1)+\wh u(\l+1)]/2.\]
Applying this to $T^*\ph$ gives
\[\wh{\C_n\,T^*\ph}(\l)=n\,\D^2\,\wh{T^*\ph}(\l)=n^{3/2}\,\D^2\,\wh\ph(\l)\]
as long as $\l\pm 1$ are also in $J$. If $\l$ is in $J$ but one of $\l\pm1$ is not in $J$ then $\l$ is near an end-point of $J$ and the error committed will be rapidly decreasing as $n\to\iy$. For the right side is $n^{3/2}$ times a linear combination of integrals like
\[\int_{-\sqn/2}^{\sqn/2} e^{\pm\pi ix\sqn}\,\ph(x)\,dx,\]
and integration by parts many time shows this is rapidly decreasing. For the left side, if for example $\l+1\not\in J$ then it is the same as the value at $\l+1-n\in J$, which is rapidly decreasing.
So we ignore this little error and use  
\[\D^2\,\wh\ph(\l)=\int_{-\sqn/2}^{\sqn/2} e^{-2\pi ix\l/\sqn}\,\C(x/\sqn)\,\ph(x)\,dx.\]
We also have $\wh{T^*x^2\ph}(\l)=\sqn\,\wh{x^2\ph}(\l)$ for $\l\in J$. Thus (ignoring the error),
\[\wh{\C_n\,T^*\ph}(\l)-2\pi^2\,\wh{T^*x^2\ph}(\l)=\sqn\,\int_{-\sqn/2}^{\sqn/2} e^{-2\pi ix\l/\sqn}\,(n\,\C(x/\sqn)-2\pi^2 x^2)\,\ph(x)\,dx.\]
From this we get 
\[\|\wh{\C_n\,T^*\ph}-2\pi^2\,\wh{T^*x^2\ph}\|\le n^{3/4}\,\|(n\,\C(x/\sqn)-2\pi^2 x^2)\,\ph(x)\|=O(n^{-1/4}),\]
since $n\,\C(x/\sqn)-2\pi^2 x^2=O(x^4/n)$ and $x^4\,\ph(x)\in L^2(\R)$. Then we get from (\ref{norm2})
\[\|\C_n\,T^*\ph-2\pi^2\,T^*x^2\ph\|=O(n^{-1}).\]
Equivalently,
\[\|T\,\C_n\,T^*\ph-2\pi^2\,TT^*x^2\ph\|=O(n^{-1}).\]
Since $TT^*\to I$ strongly, this gives
\[\|T\,\C_n\,T^*\ph-2\pi^2\,x^2\ph\|\to0.\]
Since the operator corresponding to (\ref{m2}) is multiplication by $\C_n/2$, this completes the proof of Proposition \ref{HP2}.
\end{proof}
Now we go back to the eigenvalues of $\Mnt$ and easy consequences of Propositions \ref{HP1} and \ref{HP2}.

\begin{proposition}\label{HP3} (a) If $\la_n$ are eigenvalues of $\Mnt$ and $\la_n\to\mu$, then $\mu$ is an eigenvalue of~$L$. (b) Any eigenvalue $\m$ of $L$ has a neighborhood that contains at most one eigenvalue (counting multiplicity) of $\Mnt$ for sufficiently large $n$.\sp
\end{proposition}
\begin{proof}[Proof of (a)] Suppose that $u_n$ is an eigenvector of $\Mnt$ of norm one with eigenvalue $\la_n$. In particular $(\Mnt\,u_n,u_n)=\la_n$. By Proposition \ref{HP1} there is a subsequence of $\{Tu_n\}$ that converges strongly to some $f\in L^2(\R)$.
For a Schwartz function $\ph$ we have
\[
(\mu f,\ph)=\lim\,(\la_n\,Tu_n,\ph)=\lim\,(T\Mnt u_n,\ph)\]
\[=\lim\,(u_n,\Mnt T^*\ph)=\lim\,(Tu_n,T\Mnt T^*\ph),\]
since $T$ is an isometry. By \mbox{Proposition~\ref{HP2}}, $T\Mnt T^*\ph$ converges strongly to $L\ph$. Therefore\footnote{For this we need only weak convergence of one and strong convergence of the other. But we also need that $f\ne0$, which is no easier to show than strong convergence of $Tu_n$, and we shall need strong convergence of $T\Mnt T^*\ph$ for Proposition~\ref{HP4}.} the limit equals $(f,L\ph)$, and we have shown
\[\mu (f,\ph)=(f,L\ph).\] 

It follows that $f$ is an eigenfunction of $L$ with corresponding eigenvalue $\mu$. Here is why. The eigenfunctions of $L$ are the harmonic oscillator wave functions $\ph_i$, and therefore Schwartz functions. If the corresponding eigenvalues are $\mu_i$, then  
\[(\mu-\mu_i)\,(f,\ph_i)=\mu\,(f,\ph_i)-(f,L\ph_i)=0.\] 
Since the $\ph_i$ are complete and $f\ne0$, some $\mu-\mu_k=0$. And $f$, being orthogonal to the $\ph_i$ with $i\ne k$, must be a multiple of $\ph_k$ and therefore a corresponding eigenfunction.\sp
\end{proof}
\begin{proof}[Proof of (b)] Suppose the contrary were true. Then there would be sequences of eigenvalues $\{\la_n\}$ and $\{\la'_n\}$ of $\Mnt$, both converging to $\mu$, and corresponding orthogonal (since $\Mnt$ is self-adjoint) eigenfunctions $u_n$ and $u_n'$. The strong (sub)limits $f$ and $f'$ of $Tu_n$ and $Tu_n'$ would be mutually orthogonal eigenfunctions of $L$ corresponding to the same eigenvalue $\mu$ of $L$. Since the eigenvalues of $L$ are simple, this cannot happen.\sp
\end{proof}
\begin{proposition}\label{HP4} For each eigenvalue $\mu$ of $L$ there is a sequence of eigenvalues $\la_n$ of $\Mnt$ that converges to $\mu$.\end{proposition}

\begin{proof}With corresponding eigenfunction $\ph$ of $L$ we have, by Proposition 2, 
\[\|T\Mnt T^*\ph-\mu\,\ph\|=\|T\Mnt T^*\ph-L\,\ph\|=o(1).\]
From this, and that $\|\ph-T\,T^*\,\ph\|=o(1)$ by Lemma \ref{HL2}b, 
we get
\[\|T\,(\Mnt T^*\ph-\mu\,T^*\ph)\|=o(1).\]
Since $T$ is an isometry this is the same as
\[\|\Mnt T^*\ph-\mu\,T^*\,\ph\|=o(1).\]
Since $\|T^*\ph\|=\|T\,T^*\ph\|\to\|\ph\|\ne0$, and the other eigenvalues of $L$ are bounded away from $\mu$, this implies that $\mu$ is within $o(1)$ of an eigenvalue of $\Mnt$ and $T^*\,\ph$ within $o(1)$ of an eigenvector.  See \cite{DPG}.
\end{proof}
Combining Propositions \ref{HP3} and \ref{HP4} shows that the $k$th largest eigenvalue of $M_n$ equals $1-\mu_k/n+o(1/n)$.
\subsection{The bottom eigenvalues of $M_n$}\label{HS2}

We shall find a unitary operator $U$ on $L^2(\Z_n)$ such that the quadratic form for $n(I+U_nM_nU_n^*)$ is the same as for $n(I-M_n)$ when $n$ is even and close to it when $n$ is odd. From that it will follow that the $k$th bottom eigenvalue of $M_n$ equals $-1+\mu_k/n+o(1/n)$.

Recall that Lemma \ref{HL3} says that
\[(n(I-M_n)\,u,\,u)=n^{-1/2}\,\|S(\l/n)\,\wh u\|^2+n\,\|S(k/n)\,u(k)\|^2=:Q(u),\]
where $S(\xi)=\sin(\pi\xi)$. For this we used the identity
$1-\cos \xi=2 \sin^2(\xi/2)$.
For $n(I+M_n)$ this gets replaced by $1+\cos \xi=2 \cos^2(\xi/2)$. So now we define 
\[C(\xi)=\cos(\pi\xi),\]
and get
\[(n(I+M_n)\,u,\,u)=n^{-1/2}\,\|C(\l/n)\,\wh u\|^2+n\,\|C(k/n)\,u(k)\|^2.\]

We consider first the less straightforward case of $n$ odd and define
\[v(k)=e^{2\pi i\a k}\,u\(k-(n+1)/2\),\]
where $\a$ (real) will be determined below.
We have
\[\|C(k/n)\,u(k)\|=\|C({k/n+1/2+1/2n})\,v(k)\Big\|=\|S(k/n+1/2n)\,v(k)\|.\]

Next,
\[\wh v(\l)=\sum_k e^{-2\pi ik\l/n}\,e^{2\pi i\a k}\,u\(k-(n+1)/2\)\]
\[=\sum_k e^{-2\pi i(k+(n+1)/2)\l/n}\,e^{2\pi i\a (k+(n+1)/2)}\,u(k)\]
\[=(-1)^\l\,e^{-\pi i\l/n}\,e^{\pi i\a (n+1)}\,\sum_k e^{-2\pi i k\l/n}\,e^{2\pi i\a k}\,u(k)\]
We choose $\a=-(n+1)/2n$. The factor outside the sum has absolute value 1 while the sum becomes $\wh u(\l+(n+1)/2)$. Alternatively,
\[|\wh u(\l)|=|\wh v(\l-(n+1)/2)|.\]
Therefore
\[\|C(\l/n)\,\wh u(\l)\|=\|C(\l/n+1/2+1/2n)\,\wh v(\l)\|=\|S(\l/n+1/2n)\,\wh v(\l)\|.\]

The map $U:u\to v$ is unitary and we have shown
\sp

\begin{lemma}\label{HL5} For $n$ odd we have,
\[(n(I+UM_nU^*)\,v,\,v)=n^{-1/2}\,\|S(\l/n+1/2n)\,\wh v(\l)\|^2+n\,\|S(k/n+1/2n)\,v(k)\|^2.\]
\end{lemma}

\begin{remark}. When $n$ is even we replace the shift $(n+1)/2$ by $n/2$ and $\a$ by $-1/2$, and the extra $1/2n$'s do not appear in the arguments of the $S$'s. The quadratic form becomes $Q(v)$ exactly, so $U_nM_nU_n^*=-M_n$.\footnote{This is easy to see directly.} 
\end{remark}

\begin{proposition}\label{HP5} 
$(n(I+UM_nU^*)\,v,\,v)=(1+O(n^{-1/2}))\,Q(v)+O(n^{-1/2}\|v\|^2)$.
\end{proposition}
\begin{proof} Since $dS/d\xi$ is bounded,
\[\|S(k/n+1/2n)\,v(k)\|\le\|S(k/n)\,v(k)\|+O(n\inv \|v(k)\|).\]
It follows from the arithmetic-geometric mean inequality that for any $\ep>0$
\[|a+b|^2\le (1+\ep)\,|a|^2+(1+\ep\inv)\,|b|^2.\]
We will take $\ep\to0$ as $n\to\iy$, so we obtain
\[\|S(k/n+1/2n)\,v(k)\|^2\le(1+\ep)\,\|S(k/n)\,v(k)\|^2+O(\ep^{-1}n^{-2}\|v(k)\|^2).\]
Similarly,
\[\|S(\l/n+1/2n)\,\wh v\|^2\le(1+\ep)\,\|S(\l/n)\,\wh v\|^2+
O(\ep^{-1}n^{-2}\|\wh v)\|^2)\]
\[=(1+\ep)\,\|S(\l/n)\,\wh v\|^2+O(\ep^{-1}n^{-1/2}\|v\|^2),\]
where we used $\|\wh v\|=n^{3/4}\,\|v\|$.

Thus, 
\[(n(I+UM_nU^*)\,v,\,v)\le (1+\ep)\,Q(v)+O(\ep^{-1}n\inv\|v\|^2).\]
Similarly,
\[(n(I+UM_nU^*)\,v,\,v)\ge (1+\ep)^{-1}\,Q(v)-O(\ep^{-1}n\inv\|v\|^2).\]
We set $\ep=n^{-1/2}$ and put the inequalities together to get the statement of the proposition.\end{proof}

Recall that $Q(v)=(n(I-M_n)\,v,v)$. If in the statement of the proposition we take the minimum of both sides over all $v$ with $\|v\|=1$ we deduce that 
\[n+n\la_n=(1+O(n^{-1/2}))(n-n\la_1)+O(n^{-1/2}),\]
where $\la_n$ is the bottom eigenvalue of $M_n$ and $\la_1$ the top eigenvalue. Since $\la_1=1-\m_1/n+o(1/n)$, we have $n-n\la_1=\m_1+o(1)$, and then $n+n\la_n=\m_1+o(1)$, and then $\la_n=-1+\m_1/n+o(1)$. 

Using the minimax characterization of the eigenvalues we show similarly that $\la_{n-k+1}=-1+\m_k/n+o(1)$ for each $k$.
\sp

\section{A Stochastic Argument} \label{PS5}
This section gives a bound on the largest eigenvalue of the matrix $M_n$ using a probabilistic argument.  By inspection,
$$M_n'=\frac{1}{3}I+\frac{2}{3}M_n$$ is a sub-stochastic matrix (with non-negative entries and row sums at most 1).  Take $M_n''$ as in (\ref{Peq2}), an $(n+1)\times (n+1)$ stochastic matrix corresponding to a Markov chain absorbing at 0. The first (Dirichlet) eigenvector has first entry 0 and its corresponding eigenvalue $\beta^*$ is the top eigenvalue of $M_n'$.  Thus $$\beta=\frac{\left(\beta^*-\frac{1}{3}\right)}{(2/3)}=\frac{3}{2}\beta^*-\frac{1}{2}$$
 is the top eigenvector of $M_n$.

We will work in continuous time, thus, for any transition matrix $M$,
$$M_t=\sum_{j=0}^\infty \frac{e^{-t}M^jt^j}{j!}=e^{t(M-I)}.$$

The matrix $L=I-M$, the opposite of the generator of the semigroup $\{M_t\}_{t \geq 0}$, has row sums zero, and non-positive off diagonal entries.\footnote{Since some of our readers (indeed some of our authors) may not be probabilists we insert the following note; given any matrix $L(x,y)$ with row sums zero and non-positive off diagonal entries one may construct  a continuous time Markov process $W=(W_t)_{t\geq 0}$ as follows.  Suppose $W_0=w_0$ is fixed.  The process stays at $w_0$ for an exponential time $\sigma_0$ with mean $1/|L(w_0,w_0)|$.  (Thus $P\{\sigma_0\geq t\}=e^{-t L(w_0,w_0)}$.)  Then, choose $w_1\neq w_0$ with probability 
$|L(w_0,w_1)|/L(w_0,w_0)$.  Stay at $w_1$ for an exponential time $\sigma_1$ (with mean $1/L(w_1,w_1)$).  Continue, choosing from $L(w_1,.)/L(w_1,w_1)$.}  If $\mathbf{v}$ is a right eigenvector of $M$ with eigenvalue $\beta$, then $\mathbf{v}$ is an eigenvector of $L$ with eigenvalue $1-\beta$.  A lower bound for the non trivial eigenvalues of $L$ gives an upper bound for the eigenvalues of $M$.  Throughout, we specialize to $L=I-M_n$, let $\lambda^*$ be the lowest non-zero eigenvalue of $L$, and $\beta$ the highest eigenvector of $M_n$.

Standard theory for absorbing Markov chains with all non-absorbing states connected shows that if $\tau$ is the time to first absorption, for any non absorbing state $\xi$, as $t$ tends to infinity,   $$\lim_{t\rightarrow \infty}\frac{-\log P_\xi (\tau>t)}{t}=\lambda^*.$$

Thus an upper bound on $\beta$ will follow from an upper bound on $P_\xi(\tau>t)$.  Here is an outline of the proof.  Begin by coupling the absorbing chain of interest with a simple random walk on $\mathcal{C}_n=\mathbb{Z}/(n\mathbb{Z})$.  For a fixed $b$, let $\tau_b$ be the first time that the simple random walk travels $\pm b$ from its start.  We derive the bound  $P_{\xi}(\tau>\tau_b)\leq G_b$, where $G_b<1$ is a particular constant described below.  Define a sequence of stopping times $\tau_b^i$ as follows.  $\tau_b^1=\tau_b$, $\tau_b^2$  is the first time following $\tau_b^1$ that the walk travels $\pm b$, similarly define $\tau_b^m$.  By the strong law of large numbers, $\tau^m_b/m\rightarrow \mu_b=E(\tau_b^1)$ almost surely.  Thus $$P\{\tau>m\mu_b\}\cong P\{\tau>\tau_b^m\}.$$
Using the Markov property, $P\{\tau>\tau_b^m\}\leq G^m_b.$  This implies there are positive $c_1$, $c_2$ with 
$$P\{\tau>c_1m\mu_b\}\leq G^m_b+e^{-c_2m}.$$
In our problem, classical random walk estimates show $\mu_b\sim b^2$.  We show, for $b=\sqrt{n}$, $G_{\sqrt{n}}$ is bounded away from one.  Thus $$P\{\tau>c_1m\mu_b\}\leq 2\max (G^m_b,e^{-c_2m})$$ and $$\frac{\log{P\{\tau>c_1m \mu_b\}}}{m\mu_b}\leq \frac{c'}{\mu_b}=\frac{c''}{n}$$ for some $c',c''<0$.  Backtracking gives the claimed bound in Theorem \ref{PT3}.

The argument is fairly robust\textemdash it works for a variety of diagonal entries.  At the end of the proof, some additions are suggested which should give the right constant multiplying $\frac{1}{n}$.  

We begin by constructing two processes.  For as long as possible, general absorption rates will be used.  Let $X=(x_t)_{t\geq 0}$ be the standard continuous time random walk on $\mathbb{Z}$ with jump rates $1$ between neighbors.  Take $x_0=0$.  Fix $b\in \mathbb{Z}_+$ and let $\tau_b$ be the first hitting time of $\{-b-1,b+1\}$:
\begin{equation}\label{PLeq1} \tau_b=\inf\{t\geq 0:|x_t|=b+1\}\end{equation}
Let $\{u_x\}_{x\in \mathbb{Z}} $ be killing rates, e.g.\ arbitrary non-negative real numbers.  Add a cemetery state $\infty$ to $\mathbb{Z}$.  
An absorbed process $\overline{x}=(\overline{x}_t)_{t\geq 0}$, behaving as $x$ until it is absorbed at $\infty$ with the rates $\{u_x\}_{x\in \mathbb{Z}}$ can be constructed as follows:  Let $\mathcal{E}$ be an independent exponential random variable with mean 1.  
Define an absorption time $\overline{\tau}\in [0,\infty]$ by $$\overline{\tau}=\inf_{t\geq 0}\left\{\int_0^t \mu_{x_s}ds\geq \mathcal{E}\right\}.$$
As soon as $\{u_x\}_{x\in\mathbb{Z}}$ does not vanish identically, $\overline{\tau}$ is characterized by \begin{equation}\label{PLeq2}\int_0^{\overline{\tau}} u_{x_s}ds=\mathcal{E}. \end{equation}
More simply, $\overline{x}_t=\begin{cases}x_t &\text{ if }t<\overline{\tau}\\\infty &\text{ otherwise}\end{cases}$ for $0\leq t<\infty$.

The two processes are defined on the same probability space as are $\overline{\tau}$ and $\tau_b$.  The first goal is to estimate the chance that $\overline{\tau}>\tau_b$ in terms of the given rates.  Our bounds are crude but suffice for Theorem \ref{PT3}.

\begin{proposition}\label{PLp1}  With notation as above, for any $b\geq1$,
\begin{equation}\label{PLeq3} P\{\overline{\tau}>\tau_b\}\leq \left(\frac{1}{1+(b+1)^2 v_0/2} \prod_{k=1}^b\frac{1}{1+(b+1)(b+1-k)v_k}\right)^{\frac{1}{b+1}}
\end{equation}
with $v_k=\min(u_{-k},u_k)$.
\end{proposition}
Note that the bound is achievable; if all $v_y=0$ then both sides equal 1.
\begin{proof}For any $k\in \mathbb{Z}$, $v_k\leq u_k$.  Thus if $\tau$ is the stopping time defined in (\ref{PLeq2}) with $u_k$ replaced by $v_k$, $\tau\geq \overline{\tau}$.  Therefore it is sufficient to bound $P\{\tau>\tau_b\}$ from above.  Now, everything is symmetric about zero.  Consider the process $Y=(Y_t)_{t\geq0}=(|x_t|)_{t\geq0}.$  This is Markov with jump rates:
$$J(y,y')=\begin{cases}2 &\text{ if } y=0,y'=1\\
1 &\text{ if } y=\mathbb{Z}_+\text{ and } |y'-y|=1\\
0&\text{ otherwise}.
\end{cases} $$
Clearly $\tau_b=\inf_{t\geq0}\{Y_t=b+1\}$.  Define the family of local times associated to $Y$:
$$L_y(t)=\int_0^t\delta_y(Y_s)ds \text{ for } y\in \mathbb{Z}_+,\, t\geq0,$$ where $\delta_y$ is the indicator function of $y$.
For any $t\geq 0$,
$$\int_0^t v_{x_s ds}=\sum_{0\leq y\leq b}v_yL_y(t).$$
This gives $$\{\tau>\tau_b\}=\{\sum_{0\leq y\leq b}v_yL_y(\tau_b)<\mathcal{E}\}.$$
Taking expectations of both sides with respect to $\mathcal{E}$
\begin{align}\label{PLeq4}P\{\tau>\tau_b\}&=E\left\{\operatorname{exp}\left(-\sum_{0\leq y\leq b}v_yL_y(\tau_b)\right)\right\}\\&\leq \prod_{y=1}^bE\left\{\operatorname{exp}\left(-(b+1)v_yL_y(\tau_b)\right)\right\}^{\frac{1}{b+1}}.\end{align}

The last bound follows from H{\"o}lder's inequality (with $b+1$ functions).

It is well known (see \cite{Kent} or Claim 2.4 of \cite{PS} for the discrete time version) that for any $y$, $1\leq y\leq b$, $L_y(\tau_b)$ is distributed as an exponential random variable with mean $(b+1-y)$ and $L_0(\tau_b)$ is exponential with mean $\frac{b+1}{2}$.  (The process leaves zero twice as fast as it leaves other points.)  Thus, for $1\leq y\leq b$, $$E\{\operatorname{exp}(-(b+1)v_yL_y(\tau_b)\}=\frac{(b+1-y)^{-1}}{(b+1-y)^{-1}+(b+1)v_y}.$$
$$E\{\operatorname{exp}(-(b+1)v_0L_0(\tau_b)\}=\frac{1}{1+(b+1)^2v_0/2}.$$ 

This completes the proof of Proposition \ref{PLp1}.

\end{proof}

The bound of Proposition \ref{PLp1} suggests introducing functions $F_b,G_b$ on $\mathbb{R}^{b+1}_+$.  Given by
\begin{align}\label{PLeq5}F_b(\mathbf{v})&=\left\{\frac{1}{1+(b+1)^2v_0/2}\prod_{l=1}^b\frac{1}{1+(b+1)(b+1-i)v_l}\right\}^{\frac{1}{l+1}}\\
G_b(\mathbf{v})&=\left\{\prod_{k=0}^{b}\frac{1}{1+(b+1)(b+1-k)v_k/2}\right\}^{\frac{1}{(b+1)}}.
\end{align}
They have the following crucial monotonicity properties: say that $\mathbf{v},\mathbf{v}'\in \mathbb{R}^{b+1}$ satisfies $\mathbf{v}\leq\mathbf{v}'$ if this is true coordinate-wise.   For $\mathbf{v}\in \mathbb{R}^{b+1}$, let $\overline{\mathbf{v}}$ be the non-decreasing rearrangement of $v$.  Then 
\begin{align}
\label{PLeq6} F_b(\mathbf{v})&\leq G_b(\mathbf{v})\\
\label{PLeq7}\mathbf{v}\leq\mathbf{v}'\,\Rightarrow G_b(\mathbf{v})&\geq G_b(\mathbf{v}')\\
\label{PLeq8}G_b(\mathbf{v})&\leq G_b(\overline{\mathbf{v}})
\end{align}

Return now to the process underlying Theorem \ref{PT3} (still keeping the extinction rates general.)  Let $z=(z_t)_{t\geq0}$ be defined on $\mathbb{Z}/n\mathbb{Z}$; it jumps to nearest neighbors at rate l and is killed with rates $u=(u_\xi)_{\xi\in \mathbb{Z}/n\mathbb{Z}}$.  Suppose $z_o=\xi$.  Let $v_p$, $0\leq p\leq n-1$ denote the non-decreasing rearrangement of $\mathbf{u}$.  Let $\tau$ be the absorption time of $z$.  Fix $b$, $0\leq b\leq n/2-1$ and let 
$$\tau_b=\inf_{t>0}\{z_\tau\in \{\xi-b-1,\xi+b+1\}\}.$$
Proposition \ref{PLp1} in conjunction with properties (\ref{PLeq6}), (\ref{PLeq7}), and (\ref{PLeq8}), imply that for any $\mathbf{u}$, with $G_b(\mathbf{u})$ depending only on the first $b$ coordinates of $\mathbf{u}$,
\begin{equation}\label{PLeq9} P_\xi[\tau>\tau_b]\leq G_b(\mathbf{u}).
\end{equation}
Note that the upper bound is independent of $\xi$.

Introduce a sequence $\xi_i$ of further stopping times: $\xi_1=\xi_b$, and if $\xi_m$ has been constructed,
\begin{equation}\label{PLeq10}\xi_{m+1}=\inf\{t>\xi_m: z_t\in \{z_{\xi_m}-b-1,z_{\xi_m}+b+1\}\}\end{equation}
Informally speaking these stopping times end up being good.  Because they cannot be larger than $\tau$, as in the previous treatment of a random walk on $\mathbb{Z}/n\mathbb{Z}$ coinciding with $z_t$ up to the absorption time, then $\xi_m$ are (almost surely) finite for all $m$ and the strong law of large numbers gives:
$$\lim_{m\rightarrow \alpha}\frac{\xi_m}{m}\rightarrow\mu_b=(b+1)^2$$
where $\mu_b=E\{\xi_1\}=(b+1)^2$ from the Classical Gambler's Ruin (see Chapter 14 of \cite{Feller}).

This suggests that, for $m$ large, the quantities $$P_\xi[\tau>m\mu_b]\text{ and }P_{\xi}[\tau>\xi_m]$$
should behave similarly.  Of course, care must be taken because $\tau$ and $\xi_m$ are not independent.  To proceed, we use a large deviations bound for $\xi_m$.
\begin{proposition}For $\xi_m$ defined in (\ref{PLeq10}), there are positive constants, $c_1$, $c_2$, independent of $b$ and $n$ such that for all $m\geq 1$,
$$P[\xi_m>c_1m\mu_b]\leq e^{-c_2m}.$$
\end{proposition}
\begin{proof}Observe first that this is simply a large deviations bound for the first hitting time of the simple random walk (\ref{PLeq1}) so that $n$ does not enter.  The law of $\xi_1$ is well known (see \cite{Keilson}, \cite{DM}, and \cite{Fill}).  It can be represented as a sum of $b+1$ independent exponential variables with means $a_1,a_2,\dots,a_{b+1}$  given by:
$$a_k^{-1}=2\left(1-\cos\left(\frac{\pi(2k-1)}{2(b+1)}\right)\right).$$
Thus for $\theta \in (0,a_1)$
$$E[e^{\theta\xi_1}]=\prod_{k=1}^{b+1}\frac{a_k}{a_k-\theta}.$$
By simple calculus, there is $c>0$ such that for all $a\in (0,\frac{1}{2}],\, -\log(1-a)\leq c a$.  Thus, for $\theta \in (0,a_1/2]$,
$$E(e^{\theta\xi_1})\leq e^{c\theta \sum_{n=1}^{b+1}\frac{1}{a_n}}.$$
Taking $\theta=\frac{a_1}{2}$,
$$E[e^{a_1 \xi_1/2}]\leq e^{c\sum_{n=0}^{b+1}\frac{a_1}{2a_n}}.$$
Note that $a_n$ is of order $(n/b)^2$ so the right side of the last inequality is bounded uniformly in $b$, say by $k>1$.  Now, $\xi_n$ is a sum of $m$ i.i.d.\ random variables so for any $c_1>0$
$$P[\xi_m>c_1m\mu_b]\leq e^{\frac{-c_1a_1m \mu_b}{2}}E[e^{a_1\xi_1/2}]^m\leq e^{-m(c_1a_1\mu_b)/2-\log k}.$$
Since $\mu_b=(b+1)^2$, $a_1\mu_b$ can be bounded below by a constant $\epsilon>0$, uniformly in $b\in \mathbb{N}$.  Thus if $c_1=4\log k/\epsilon$, the claimed bound holds with $c_2=\log(k)$.
\end{proof}
We can now set up a bound for the top eigenvalue.  Working on $\mathbb{Z}/n\mathbb{Z}$ but still with general absorption rates:
\begin{align}
P_\xi[\tau>c_1m\mu_b]&=P_\xi[\tau>c_1m\mu_b;\xi_m\leq c_1m \mu_b]+P_\xi[\tau>c_1m\mu_d,\xi_m>c_1m\mu_b]\\
&\leq P_\xi[\tau>\xi_m]+P_\xi[\xi_m>c_1m\mu_b]\\
&\leq G_b^m(\mathbf{v})+e^{-c_2m}.
\end{align}
It follows that
$$\lambda^*=-\lim_{m\rightarrow\infty}\frac{1}{c_1m\tau_b}\log(P_\xi[\tau>c_1m\tau_b])\geq \frac{1}{c_1\tau_b}\min \{c_2,-\log(G_b(x))\}.$$
Since $\mu_b=(b+1)^2$, proving that with $b$ of order $\sqrt{n}$, $-\log(G_b(x))$ is bounded below by a positive constant, uniformly in $n$, will complete the proof.

Up to now, the kill rates $\mathbf{u}$ have been general.  Specialize now to the rates for the matrix $M''$ with any scrambling of its diagonal.  The vector $\overline{\mathbf{v}}$ is given by the $b+1$ entries of:
\begin{align*}0,\frac{1}{3}\left(1-\cos\left(\frac{2\pi}{n}\right)\right)&,\frac{1}{3}\left(1-\cos\left(\frac{2\pi}{n}\right)\right),\frac{1}{3}\left(1-\cos\left(\frac{4\pi}{n}\right)\right),\\&\frac{1}{3}\left(1-\cos\left(\frac{4\pi}{n}\right)\right),\dots, \frac{1}{3}\left(1-\cos\left(\frac{2\pi\left\lfloor\frac{n}{2}\right\rfloor}{n}\right)\right)\end{align*}

From the definition of $G_b$ at (\ref{PLeq5}) with $b=\lfloor\sqrt{n}\rfloor$, a Riemann sum approximation gives
$$\lim_{n\rightarrow \infty}G_b(\mathbf{v})=e^{\frac{-\pi^2}{24}}.$$
Indeed, 
\begin{align*}
-\log(G_b(\mathbf{v}))=&\frac{1}{b+1}\sum_{y=0}^{b-1} \log\left(1+(b+1)(b+1-y)\left(1-\cos\left(\frac{2\pi}{n}\left\lfloor\frac{y+1}{2}\right\rfloor\right)\right)\right)\\
&\sim\frac{2\pi^2}{n^2}\frac{1}{b}\sum_{y=0}^{b-1}(b+1)(b+1-y)\left\lfloor \frac{y+1}{2}\right\rfloor^2\\
&\sim \frac{\pi^2b^4}{2n^2}\frac{1}{b}\sum_{y=0}^{b-1}\left(1-\frac{y}{b}\right)\left(\frac{y}{b}\right)^2\\
& \sim\frac{\pi^2b^4}{2n^2}\int_0^1(1-y)y^2 dy\\
&\sim \frac{\pi^2}{24}.
\end{align*}
Combining the pieces, we use $\beta$ for the highest eigenvalue of $M_n$ and thus $\beta=\frac{3}{2}\beta^*-\frac{1}{2}$. Using this notation, we have shown that $\frac{c}{n}\leq \lambda^*=1-\left(\frac{1}{3}+\frac{2}{3}\beta\right).$   Thus $\beta\leq 1-\frac{3c}{2n}$.  This completes the argument and ends the proof of Theorem \ref{PT3}.
\begin{remarks}The above argument can be modified to handle quite general diagonal elements (in particular $\cos\left(\frac{2\pi a j}{n}\right), 0\leq j\leq N-1$, needed for the application to the Heisenberg random walk).  Indeed, for $a=o(n)$, the argument goes through with no essential change with $b=\sqrt{\frac{n}{a}}$ to show that with diagonal entries $\cos\left(\frac{2\pi a j}{n}\right)$, $0\leq j\leq n-1$, the eigenvalue bound $1-\frac{ca}{n}$ holds (with $c>0$ independent of $n$ and $a$).

The use of H\"older's inequality in (\ref{PLeq4}) is crude.  The joint distribution of the local times of birth and death processes is accessible (see \cite{Kent}).  We hope this can be used to give sharp results for the constant.  Finally we note that the approach to bound $\beta$ via an associated absorbing Markov chain was used in \cite{us}.  There, a geometric path argument was used to complete the analysis.  This gave cruder bounds ($\beta\leq 1- \frac{c}{n^\frac{4}{3}}$) but the argument worked for diagonal entries $\cos\left(\frac{2\pi a \xi}{n}\right)$ for any $1\leq a \leq \frac{n}{2}$ as well as negative eigenvalues.
\end{remarks}
\section{A random walk on the affine group (mod $p$)}\label{PS6}
  Let $\mathcal{A}_p$ be the affine group (mod $p$).  Here, $p$ is prime and elements of $A_p$ can be represented as pairs $(a,b), 1\leq a\leq p-1,\,\,0\leq b\leq p-1$ 
  $$(a_1,b_1)(a_2,b_2)=(a_1a_2,a_1b_2+b_1).$$
 All entries are taken mod $p$.  Fix a generator $g$ of the multiplicative group.  Let $$S=\{(1,0),(1,1),(1,-1),(g,0),(g^{-1},0)\}.$$
 Set  \begin{equation}\label{Peq4.1}Q(h)=\begin{cases}\frac{1}{5} &\text{ if }h\in S\\
 0 &\text{ else}.\end{cases}\end{equation} 
 Convolution powers of $Q$ converge to the uniform distribution $U(h)=\frac{1}{p(p-1)}.$  We use the representation theory of $A_p$ and the analytic results of previous sections to show that order $p^2$ steps are necessary and sufficient for convergence.
 \begin{theorem}  With definitions above, there are positive universal constants $c_1$,$c_2$, and $c_3$ such that for all primes $p$ and $k\geq 1$
 $$c_1e^{-c_2\frac{k}{p^2}}\leq \|Q^{*^k}-U\|_{TV}\leq c_3 e^{-\frac{c_2k}{p^2}}.$$
 \end{theorem}
 \begin{proof}By the usual Upper Bound Lemma (see \cite{Diaconis}, Chapter 3):
 $$4\|Q^{*^k}-U\|_{TV}\leq \sum_{\rho\neq 1}d_\rho\|\hat{Q}(\rho)^k\|^2.$$
 Here, the sum is over nontrivial irreducible representations $\rho$ of $A_p$, $d_\rho$  is the dimension of $\rho$, $\hat{Q}(\rho)=\sum_h Q(h)\rho(h)$ and the norm on the right is the trace norm.  There are $p-1$ one dimensional irreducible representations indexed by $\alpha\in \{1,2,\dots,p-1\}$.
 \begin{equation}\label{Peq4.2}\rho_\alpha(a,b)=e^{2\pi i \alpha \sigma(a)/(p-1)}.
 \end{equation}
 where $\sigma:\mathbf{Z}_p^*\rightarrow\mathbf{Z}_{p-1}$ is the group morphism such that $\sigma(g)=1$.
 Then $$\hat{Q}(\rho_\alpha)=\frac{3}{5}+\frac{2}{5}\cos\left(\frac{2 \pi\alpha}{p-1}\right).$$
 There is one $(p-1)$ dimensional representation $\rho$.  This may be realized on $$V=\{f:\{1,2,\dots,p-1\}\rightarrow \mathbb{C}\}$$
 with $$\rho(a,b)f(j)=e^{\frac{2\pi i jb}{p}}f(aj),\,\,\,1\leq j\leq p-1.$$
 It is easy to check directly that $\rho$ is a representation with character
 $$\chi(a,b)=\begin{cases}
 0& a\neq 1\\
 -1 & a=1,\,b\neq 0\\
 p-1&a=1,\,b=0
 \end{cases}.$$
 A further simple check shows that $\langle \chi|\ch	\rangle=\frac{1}{p(p-1)}\sum_{a,b}|\chi(a,b)|^2=1$ and that $\chi$ is orthogonal to the characters $\rho_\alpha$ in (\ref{Peq4.2}).  It follows that $\{\rho_{\alpha}\}_{\alpha=1}^{p-1},\rho$ is a full set of irreducible representations.  Choose a basis $\delta_{g^a}(\cdot)$ for $V$, $0\leq a\leq p-2$.  Then, for ${\mathcal{Q}}$ in (\ref{Peq4.1}),
 $$\hat{\mathcal{Q}}(\rho)=\frac{1}{5}\begin{tikzpicture}[baseline=(current bounding box.center)]
\matrix (m) [matrix of math nodes,nodes in empty cells,right delimiter={)},left delimiter={(} ]{
  & 1 &  & &&   & 1  \\
 1& & & & &&   \\
  & & & & &&   \\
  & & 1&\phantom{1} &1& &    \\
   & & & & &&   \\
  & & & & &&1  \\
1 & & && &  1&\phantom{1} \\
} ;
\draw[loosely dotted,thick] (m-1-1)-- (m-4-4);
\draw[loosely dotted,thick] (m-2-1)-- (m-4-3);
\draw[loosely dotted,thick] (m-1-2)-- (m-4-5);
\draw[loosely dotted,thick] (m-4-4)-- (m-7-7.center);
\draw[loosely dotted,thick] (m-4-3)-- (m-7-6);
\draw[loosely dotted,thick] (m-4-5)-- (m-6-7);
\node[align=center] (cosgen) at (4.8cm,1cm) [below=3mm]{$1+2\cos\left(\frac{2\pi j}{p}\right)$,\,\, \tiny{$1\leq j \leq p-1$}};
\path[thick, bend left=45, <-] 
 (m-4-4.center) edge (cosgen.west);
\end{tikzpicture}.$$

Using any of the three techniques above, there is a constant $c>0$ such that the largest and smallest eigenvalues of $\hat{\mathcal{Q}}(\rho)$ (in absolute value) are bounded above by $1-\frac{c}{p}.$  Combining bounds
$$4\|Q^{*^k}-U\|^2_{TV}\leq\sum_{j=1}^{p-1} 
\left(\frac{3}{5}+\frac{2}{5}\cos\left(\frac{2\pi j}{p}\right)\right)^{2k}+(p-1)^2\left(1-\frac{c}{p}\right)^{2k}.$$
Using $\cos(x)=1+\frac{x^2}{2}+O(x^4)$, the sum is at most $c_1'e^{-c_2'\frac{k}{p^2}}$ for universal $c_1'$, $c_2'$.    The final term is exponentially smaller proving the upper bound.  The lower bound follows from the usual second moment method. (See \cite{Diaconis} Chapter 3 Theorem 2 for details.)  Further details are omitted.
 \end{proof}
 
 \begin{remark}
 In this example, the matrix $\hat{\mathcal{Q}}(\rho)$ is again the sum of a circulant and a diagonal matrix.  Here, the circulant has eigenvalues $\frac{2}{5}\cos\left(\frac{2\pi j}{p-1}\right)$, $0\leq j\leq p-2$ and the diagonal matrix has entries $\frac{3}{5}+\frac{2}{5}\cos\left(\frac{2\pi j}{p}\right)$, $1\leq j\leq p-1$.  The Weyl bounds show that the largest and smallest eigenvalues are bounded in absolute value by $1-\frac{\theta}{p^2}$ for some fixed $\theta>0$.  Using this to bound the final term in the upper bound gives $(p-1)^2\left(1-\frac{c}{p^2}\right)^{2k}.$  This shows that the walk is close to random after order $p^2\log (p)$ steps.  In the Heisenberg examples the Weyl bounds give a bound of 1 which is useless.
\end{remark}
The methods above can be applied to other walks on other groups.  While we won't carry out the details here, we briefly describe two further examples and point to our companion paper \cite{BDHMW3} for more.
\begin{example}[Borel Subgroup of $SL_2(\mathbb{F}_p)$]  Let $G$ be the $2\times 2$ matrices of the form:
$$\begin{bmatrix}a&b\\0&a^{-1}\end{bmatrix}\,\,\, a\in \mathbb{F}_p^*, b\in \mathbb{F}_p\leftrightarrow (a,b).$$
A minimal generating set (with the identity) is
$$S=\{\text{id}, (g,0), (g^{-1},0),(1,1),(1,-1)\},\,\, g \text{ a generator of }\mathbb{F}_p^*.$$
The group has order $p(p-1)$ with $p-1$ 1-dimensional representations and 4 representations of dimension $(p-1)/2$.  They are explicitly described in \cite{CR} p. 67.   The Fourier analysis of the measure $Q$ supported on $S$ is almost the same as the analysis for the affine group.  The results are that order $p^2$ steps are necessary and sufficient for convergence to the uniform distribution.
\end{example}
\begin{example}[$M(p^3)$]  There are two nonabelian groups of order $p^3$: the Heisenberg group discussed above and $M(p^3)$. See \cite{Suzuki} Chapter 4 Section 4.  One description of the latter is:
$$M(p^3)=\{(a,b):a\in \mathbb{Z}_p, b\in \mathbb{Z}_p^2\}, \,\, (a,b)(a',b')=(a+a',a*b'+b)$$
with $a*b=(1+ap)b \,\,(\operatorname{mod} p^2).$  This group has the same character table as $H_1(p)$.  It thus has $p^2$ 1-dimensional representations and $p-1$ representations of dimension $p$.  A minimal generating set (for odd $p$ the identity is not needed to take care of parity problems) is $$S=\{(1,0)(-1,0)(0,1),(0,-1)\}.$$
The Fourier transforms of the associated $Q$ at the $p$ dimensional representations have the same form as the matrices in (\ref{Peq1.1}) with diagonal elements $$2\cos\left(\frac{2\pi c}{p^2}(1+jp)\right), \,\,0\leq j\leq p-1$$
where $1\leq c\leq p-1$ is fixed (for the $c$th representation).  We have not carried out the details, but, as shown in \cite{DiaconisA}, it is known that order $p^2$ steps are necessary and sufficient for convergence.
\end{example}
\section{Eigenvalues in the Bulk}
Consider the matrix $M_n(a)$ as in (\ref{Peq1.1}) with $$\cos\left(\frac{2\pi j a}{m}\right),\,\,0\leq j\leq n-1$$ as the diagonal elements.  The sections above give bounds on the largest and smallest eigenvalues.  It is natural to give bounds for the empirical measure of all the eigenvalues.  This is straightforward, using a theorem of Kac-Murdock-Szeg{\"o} from \cite{KMS}.  We use the elegant form of Trotter \cite{Trotter}.  If $\lambda_1\geq\lambda_2\geq\dots\geq \lambda_n$ are the eigenvalues of $M_n(a)$, let $$\Lambda_n=\frac{1}{n}\sum_{i=1}^n\delta_{\lambda_i }$$ be the associated empirical measure.  To describe the limit let
\begin{equation}\label{Peqstar}f_2(x)=\begin{cases}\frac{2}{\pi(1 + |x|)} F_{2, 1} \left( \frac{1}{2}, \frac{1}{2} ; 1 ;
  \left( \frac{1 - |x|}{1 + |x|} \right)^2\right)&-1\leq x\leq 1\\0& \text{ else }\end{cases}\end{equation}
where $F_{2,1}$ is the hypergeometric function.
%
Let $\mu_2$ be the associated measure.  Distance between $\Lambda_n$ and $\mu_2$ is measured in the $d_2$ Wasserstein distance:
$$d_2^2(\Lambda_n,\mu_2)=\sup E|W-Z|^2 \text{ with } W\sim \Lambda_n, \, Z\sim\mu_2.$$
\begin{theorem}  Let $\Lambda_n$ be the empirical measure of the matrix $M_n(a)$ with $1\leq a\leq n-1$.  Let $\mu_2$ be defined by (\ref{Peqstar}).   Then, with $a$ fixed, as $n\rightarrow \infty$,
$$d_2(\Lambda_n,\mu_2)\rightarrow 0.$$
\end{theorem}
See Figure \ref{AngFig1} for an example.

\begin{figure}\includegraphics[width=4in]{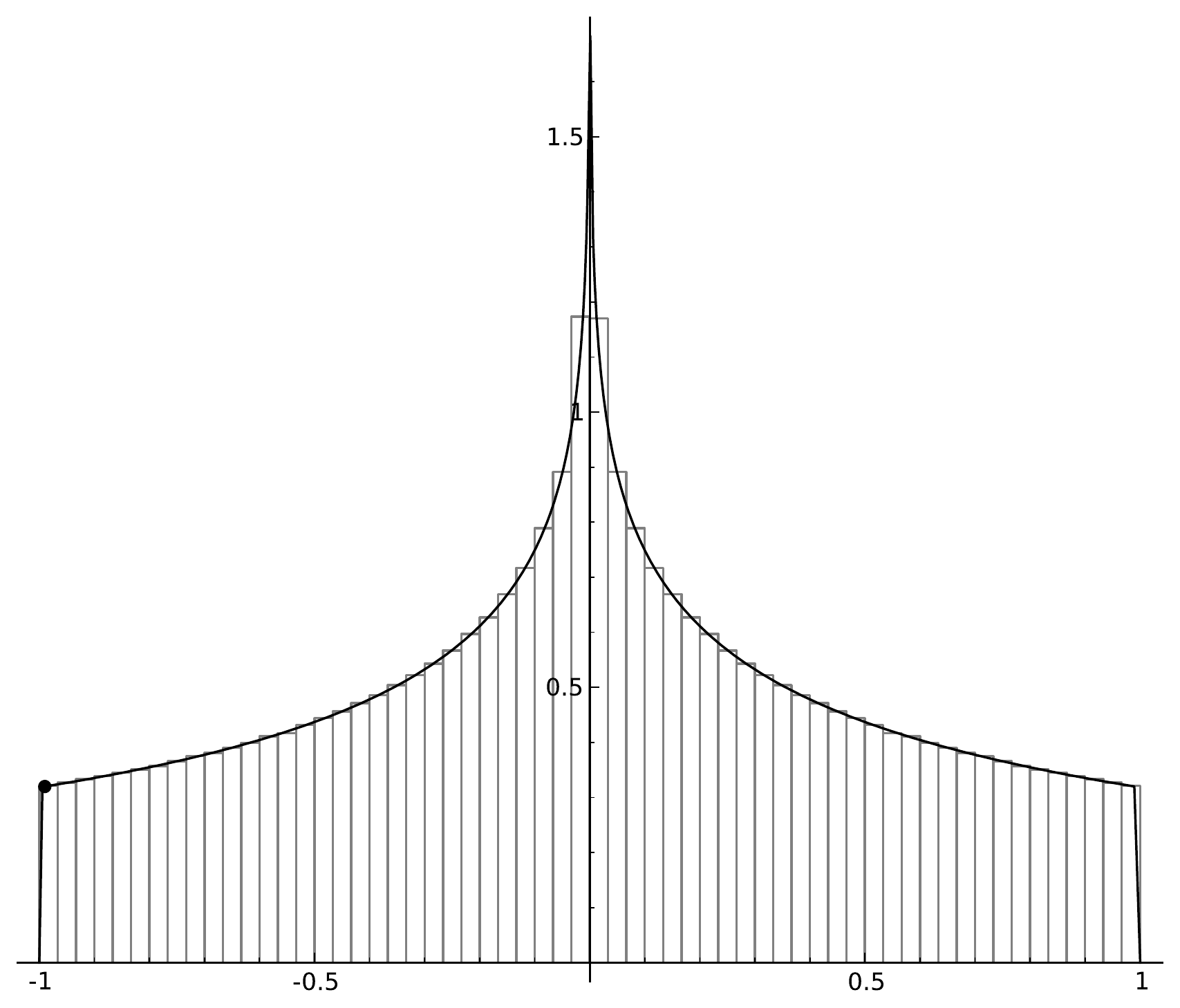}
\caption{The curve shows the eigenvalues predicted by $f_2$, while the histogram gives the distribution of the actual eigenvalues of $M_{10,000}(=M_{10,000}(1))$.  Note that the curve has a very extreme, but finite slope around $-1$ and $1$.  For example, although it is clear that $f_2(-1)=0$, the small point on the left of the picture corresponds to $(-.99,f_2(-.99))\approx .32$}\label{AngFig1}
\end{figure}
\begin{remark}  We have not seen a way to use this kind of asymptotics to bound the rate of convergence of a random walk.  Indeed our limit theorem shows that the distribution of the bulk does not depend on $a$ while previous results show the extreme eigenvalues crucially depend on $a$.
\end{remark}
\begin{proof}Trotter's version of the Kac-Murdock-Szeg{\"o} theorem applies to $M_n$.  If $$\sigma(x,y)=\cos(2\pi a x)+\cos(2\pi y)\,\, 0\leq x,y\leq 1,$$ consider $\sigma$ as a random variable on $[0,1]^2$, endowed with the Lebesgue measure.  This has distribution $\cos(2\pi a U_1)+\cos(2\pi U_2)$ where $U_1$ and $U_2$ are independent uniform on $[0,1]$.  An elementary calculation shows that $\cos(2\pi a U)$ has an arc-sine density $f(x)$ no matter what the integer $a$ is.  
\begin{equation}\label{my34}f(x)=\begin{cases}\frac{1}{\pi\sqrt{1-x^2}}&-1\leq x\leq 1\\0& \text{ else }\end{cases}.\end{equation}Trotter shows that the empirical measure is close to $\mu_2$, the distribution of $\sigma$.
It follows that the empirical measure of the eigenvalues has limiting distribution the law of $(X+Y)/2$ where $X$ and $Y$ are independent with density $f(x)$.  This convolution has density \begin{equation}f_2(x)=\begin{cases}\frac{2 }{\pi^{2}}\, \int_{\max\{-1,2x-1\}}^{\min\{1,2x + 1\}} \frac{1}{\sqrt{{\left({\left(2x -y\right)}^{2} - 1\right)} {\left(y^{2} - 1\right)}}}\,{d y}&-1\leq x\leq 1\\0& \text{ else }\end{cases}.
\end{equation}  The arguement below shows that this integral is in fact \begin{equation}f_2(x)=\begin{cases}\frac{2}{\pi(1 + |x|)} F_{2, 1} \left( \frac{1}{2}, \frac{1}{2} ; 1 ;
  \left( \frac{1 - |x|}{1 + |x|} \right)^2\right)&-1\leq x\leq 1\\0& \text{ else }\end{cases}\end{equation}

The integral in (\ref{my34}) is in fact a well known integral in a different guise.  Let $0 \leqslant k \leqslant 1$. Define
\begin{equation}
  \label{kkdef} K (k) = \int_0^1 \frac{d t}{\sqrt{(1 - t^2) (1 - k^2 t^2)}} .
\end{equation}
This is a complete elliptic integral and equals
\[ \frac{\pi}{2} F_{2, 1} \left( \frac{1}{2}, \frac{1}{2} ; 1, k^2 \right) .
\] (See Section 22.301 of \cite{WandW}.)
For ease of notation, we will evaluate 
\[ f_3(x) = \int_{\max (- 1, x - 1)}^{\min (1, x + 1)} \frac{d t}{\sqrt{(1 - t^2) (1 - (x - t)^2})}  \] for $|x| \leqslant 2$
 Making the variable change $t
\rightarrow t + \frac{x}{2}$ it becomes
\[ \int_{\max (- \frac{x}{2} - 1, \frac{x}{2} - 1)}^{\min (- \frac{x}{2} + 1,
   \frac{x}{2} + 1)} \frac{d t}{\sqrt{h (t)}} \]
where
\[ h (t) = (a^2 - t^2) (b^2 - t^2), \hspace{2em} a = 1 + \frac{x}{2},
   \hspace{1em} b = 1 - \frac{x}{2} . \]
This is an even function of $x$ so it is enough to consider when $x \geqslant 0$. Then we need to evaluate
\[ \int_{- b}^b \frac{d t}{\sqrt{h (t)}} . \]
Make the variable change $t \rightarrow b t$ and the integral becomes
\[ \int_{- 1}^1 \frac{d t}{\sqrt{(1 - t^2) (a^2 - b^2 t^2)}} = \frac{2}{a} K
   (k), \hspace{2em} k = \frac{b}{a} . \]
The factor of $2$ comes from the fact that we are integrating an even function
from $- 1$ to $1$, whereas in (\ref{kkdef}) the integral is from $0$ to $1$.
Thus

\[ f_3(x) = \frac{2 \pi}{2 + |x|} F_{2, 1} \left( \frac{1}{2}, \frac{1}{2} ; 1 ;
  \left( \frac{2 - |x|}{2 + |x|} \right)^2\right) . \] Sending $x\rightarrow 2x$ and multiplying by the appropriate constant, we have that the integral in (\ref{my34}) is in fact
  $$\frac{2}{\pi(1 + |x|)} F_{2, 1} \left( \frac{1}{2}, \frac{1}{2} ; 1 ;
  \left( \frac{1 - |x|}{1 + |x|} \right)^2\right)$$\end{proof}
\begin{remark}
\cite{PT} gives a similar expresson for the sum of two general beta variables.
\end{remark}
\begin{acknowledgment}  As this work progressed, we received useful comments from Florin Boca, Ben Bond, Bob Guralnick, Susan Holmes, Marty Isaacs, Evita Nestoridi, Jim Pitman, Laurent Saloff-Coste, and Thomas Strohmer.  We offer thanks for the remarks of our very helpful reviewer.
\end{acknowledgment}
\bibliographystyle{plain}
\bibliography{CpD}

\end{document}